\documentclass{amsart}
\usepackage[utf8]{inputenc}
\usepackage{amssymb}
\usepackage{amsmath}
\usepackage{amsthm, amsfonts, mathrsfs}
\usepackage{amsmath}
\usepackage{amssymb}
\usepackage{amsthm}
\usepackage{graphicx}

\usepackage{amsrefs}
\newtheorem{thm}{Theorem}[section]
\newtheorem{cor}[thm]{Corollary}

\newtheorem{prop}[thm]{Proposition}

\newtheorem{qu}[thm]{Question}
\newtheorem{ex}[thm]{Example}

\theoremstyle{plain}
\newtheorem{theorem}{Theorem}[section]
\newtheorem{lemma}[theorem]{Lemma}
\newtheorem{proposition}[theorem]{Proposition}

\newtheorem{definition}[theorem]{Definition}

\theoremstyle{definition}

\theoremstyle{remark}

\newtheorem{fact}{Fact}

\numberwithin{equation}{section}

\DeclareMathOperator{\val}{\mathop{val}}

\begin{document}
\title{Compact Spaces with a P-base}

\author[A. Dow]{Alan Dow}
\address{UNC Charlotte}
\email{adow@uncc.edu}

\author[Z. Feng]{Ziqin Feng}
\address{Department of Mathematics and Statistics\\Auburn University\\Auburn, AL~36849}
\email{zzf0006@auburn.edu}

\begin{abstract} In the paper, we investigate (scattered) compact spaces with a $P$-base for some poset $P$. More specifically, we prove that, under the assumption $\omega_1<\mathfrak{b}$, any compact space with an $\omega^\omega$-base is first-countable  and any scattered compact space with an $\omega^\omega$-base is countable.
These give positive solutions to Problems 8.6.9 and 8.7.7 in \cite{Banakh2019}. Using forcing, we also prove that in a model of $\omega_1<\mathfrak{b}$, there is a non-first-countable compact space with a $P$-base for some poset $P$ with calibre~$\omega_1$.  \end{abstract}

\date{\today}

\maketitle
\section{Introduction}

Let $P$ be a partially ordered set. A topological space $X$ is defined to have a neighborhood $P$-base at $x\in X$ if there exists a neighborhood base $(U_p[x])_{p\in P}$ at $x$ such that $U_p[x]\subset U_{p'}[x]$ for all $p\geq p'$ in $P$. We say that a topological space has a $P$-base if it has a neighborhood $P$-base at each $x\in X$. All topological spaces in this paper are regular.

We will use Tukey order to compare the cofinal complexity of posets. The Tukey order \cite{Tuk40} was originally introduced, early in the 20th century, as a tool to
understand convergence in general topological spaces, however it was quickly seen to have broad
applicability in comparing partial orders.  Given two directed sets $P$ and $Q$, we say $Q$ is a Tukey quotient of $P$, denoted by $P\geq_T Q$, if there is a map $\phi:P\rightarrow Q$ carrying cofinal subsets of $P$ to cofinal subsets of $Q$. In our context, where $P$ and $Q$ are both Dedekind complete (every bounded subset has the least upper bound), we have $P\geq_T Q$ if and only if there is a map $\phi: P\rightarrow Q$ which is order-preserving and such that $\phi(P)$ is cofinal in $Q$. If $P$ and $Q$ are mutually Tukey quotients, we say that $P$ and $Q$ are Tukey equivalent, denoted by $P=_T Q$. It is straightforward to see that a topological space $X$ has a $P$-base if and only if $\mathcal{T}_x(X)\leq_T P$ for each $x\in X$, here, $\mathcal{T}_x(X)=\{U: U \text{ is an open neighborhood of }x\}$.

 Topological spaces and function spaces with an $\omega^\omega$-base were systematically studied in \cite{Banakh2019}. Lots of work about the $\omega^\omega$-base in topological groups  have been done in \cite{BL18}, \cite{GKL15}, \cite{LPT17}, and \cite{SF20}. In this paper we investigate the Tukey reduction of a $P$-base in some (scattered) compact spaces with $P$ satisfying some Calibre conditions.   This paper is organized in the following way.

 In Section~\ref{compact}, we show that if $P$ has Calibre~$\omega_1$, then any compact space with a $P$-base is countable tight. Furthermore, we prove that if a compact space with countable tightness has  a $\mathcal{K}(M)$-base for some separable metric space $M$, then it is first-countable. As a corollary, any compact space with an $\omega^\omega$-base is first-countable under the assumption $\omega_1<\mathfrak{b}$. This gives a positive answer to Problem 8.7.7 in \cite{Banakh2019}. In Section~\ref{scacompact}, we address Problem 8.6.9  in \cite{Banakh2019} positively by showing that any scattered compact space with an $\omega^\omega$-base is countable under the assumption $\omega_1<\mathfrak{b}$. 
 It is natural to ask whether under the assumption $\omega_1<\mathfrak{b}$  any compact with a $P$-base is first-countable if $P$ satisfies some Calibre properties, for example, Calibre~$\omega_1$. In Section~\ref{calbreo1}, we prove that in a model of Martin's Axiom in which $\omega_1<\mathfrak{b}$, there is a non-first-countable compact space with a $P$-base for some poset $P$ with calibre~$\omega_1$.

 \section{Preliminaries}

For any separable metric space $M$, $\mathcal{K}(M)$ is the collection of compact subsets of $M$ ordered by set-inclusion. Fremlin observed that if a separable metric space $M$ is locally compact, then $\mathcal{K}(M)=_T\omega$. Its unique successor under Tukey order is the class of Polish but not locally compact spaces. For $M$ in this class, $\mathcal{K}(M)=_T\omega^\omega$ where $\omega^\omega$ is ordered by $f\leq g$ if $f(n)\leq g(n)$ for each $n\in \omega$. In \cite{GM16}, Gartside and Mamataleshvili constructed a $2^{\mathfrak{c}}$-sized antichain in $\mathcal{K}(\mathcal{M})=\{\mathcal{K}(M): M\in \mathcal{M}\}$ where $\mathcal{M}$ is the set of separable metric spaces. 

Let $P$ be a directed poset, i.e. for any points $p, p'\in P$, there exists a point $q\in P$ such that $p\leq q$ and $p'\leq q$.  A subset $C$ of $P$ is \emph{cofinal} in $P$ if for any $p\in P$, there exists a $q\in C$ such that $p\leq q$. Then $\text{cof}(P)=\min\{|C|: C\text{ is cofinal in }P\}$. We also define $\text{add}(P)=\min\{|Q|: Q \text{ is unbounded in }P\}$. For any $f, g\in \omega^\omega$, we say that $f\leq^\ast g$ if the set $\{n\in\omega:f(n)>g(n) \}$ is finite. Then $\mathfrak{b}=\text{add}(\omega^\omega, \leq^\ast)$ and $\mathfrak{d}=\text{cof}(\omega^\omega, \leq^\ast)$. See \cite{Douwen84} for more information about small cardinals.

Let $\kappa\geq \mu\geq \lambda$ be cardinals. We say that a poset $P$ has \emph{calibre~$(\kappa, \mu, \lambda)$} if for every $\kappa$-sized subset $S$ of $P$ there is a $\mu$-sized subset $S_0$ such that every $\lambda$-sized subset of $S_0$ has an upper bound in $P$. We write calibre~$(\kappa, \mu, \mu)$ as calibre~$(\kappa, \mu)$ and calibre~$(\kappa, \kappa, \kappa)$ as calibre~$\kappa$. It is known that $\mathcal{K}(M)$ has Calibre~$(\omega_1, \omega)$ for any separable metric space $M$, hence so does $\omega^\omega$. Under the assumption $\omega_1=\mathfrak{b}$, $\omega_1$ is a Tukey quotient of $\omega^\omega$. Furthermore, under the assumption $\omega_1<\mathfrak{b}$, the poset $\omega^\omega$ has Calibre~$(\omega_1, \omega_1,
\omega_1)$, i.e. Calibre~$\omega_1$.    We will use this fact in several places of this paper.

It is clear that if $P\leq_T Q$ and $Q\leq_T R$ then $P\leq_T R$ for any posets $P, Q,$ and $R$. So we get the following proposition.

\begin{prop} Let $P$ and $Q$ be posets such that $P\leq_T Q$. Then if a space $X$ has a neighborhood $P$-base at $x\in X$, then $X$  also has a neighborhood $Q$-base at $x$. Hence, any space with a $P$-base also has a $Q$-base. \end{prop}

\begin{prop}\label{sub} If $X$ has a $P$-base, then any subspace of $X$ also has a $P$-base.

\end{prop}



 \begin{prop} Let $P$ be a poset with $\omega_1\leq_T P$ and $P=_T\omega\times P$. Then the space $\omega_1+1$ has a $P$-base.
 \end{prop}

 \begin{proof} For each $\alpha<\omega_1$, the space $\omega_1+1$ has a countable local base at $\alpha$. Hence $\mathcal{T}_\alpha(\omega_1+1)\leq_T P$ due to the fact that $P=_T\omega\times P$.

 Let $\phi$ be a map from $P$ to $\omega_1$ which carries confinal subsets of $P$ to confinal subsets of $\omega_1$. Then we define a map $\psi$ from $P$ to $\mathcal{T}_{\omega_1}(\omega_1+1)$ by $\psi(p)=(\phi(p), \omega_1]$ for each $p\in P$. Clearly $\psi$ carries confinal subsets of $P$ to confinal subset of  $\mathcal{T}_{\omega_1}(\omega_1+1)$. Hence the space $\omega_1+1$ has a neighborhood $P$-base at $\omega_1$. This finishes the proof.  \end{proof}

 As a result of $\mathfrak{b}\leq_T \omega^\omega$, the space $\omega_1+1$ has an $\omega^\omega$-base under the assumption $\omega_1=\mathfrak{b}$. Gartside and Mamatelashvili in \cite{GM17} proved that $\omega^\omega\times \omega_1\leq_T\mathcal{K}(\mathbb{Q})\leq_T\omega^\omega\times [\omega_1]^{<\omega}$, here $\mathbb{Q}$ is the space of rationals. Hence, we have the following result.
  \begin{cor}\label{o1kq} The space $\omega_1+1$ has a $\mathcal{K}(\mathbb{Q})$-base.
 \end{cor}

A generalization of $G_\delta$-diagonals is $P$-diagonals for some poset $P$. A collection $\mathcal{C}$ of subsets of a space $X$ is $P$-directed if $\mathcal{C}$ can be represented as $\{C_p: p\in P\}$ such that $C_p\subseteq C_{p'}$ whenever $p\leq p'$. We say $X$ has a $P$-diagonal if $X^2\setminus \Delta$ has a $P$-directed compact cover, where $\Delta=\{(x, x): x\in X\}$. The second author showed that any compact space with a $\mathcal{K}(\mathbb{Q})$-diagonal is metrizable in \cite{Feng2019} and S\'{a}nchez proved that the same result holds for any compact space with a  $\mathcal{K}(M)$-diagonal for some separable metric space $M$ in \cite{San2020}. Here, we include two results about spaces with (or without) $P$-diagonal giving that $P$ satisfies some Calibre properties.

\begin{prop}Let $P$ be a poset with Calibre~$(\omega_1, \omega)$. The space $\omega_1+1$ doesn't have a  $P$-diagonal. \end{prop}

\begin{proof}Suppose that $\omega_1+1$ has a $P$-diagonal, i.e., a $P$-ordered compact covering $\{K_p:p\in P\}$ of $(\omega_1+1)^2\setminus \Delta$ Choose $\alpha_\gamma$  and $\beta_\gamma$ in $\omega_1$ for $\gamma\in \omega_1$ such that

$$\alpha_0 < \beta_0 <\alpha_1 <\beta_1< \ldots < \alpha_\gamma <\beta_\gamma <\dots.$$

Let $p_\gamma$ in $P$ be such that $(\alpha_\gamma,\beta_\gamma)\in K_{p_\gamma}$ for each $\gamma\in \omega_1$.   By Calibre $(\omega_1, \omega)$, there are $p$ in $P$  and $\gamma_n\in \omega_1, n\in\omega$  such that $\gamma_0<\gamma_1<\ldots$   and $K_{p_{\gamma_n}} \subset K_p$ for each $n\in\omega$.  Then $(\delta,\delta)\in K_p$, where $\delta = \sup \{\alpha_{\gamma_n}: n\in\omega \}$. This contradiction finishes the proof. \end{proof}

\begin{prop}Let $P$ be a poset with Calibre~$\omega_1$. Any compact space with a $P$-diagonal has countable tightness. \end{prop}

\begin{proof} Let $\{K_p:p\in P\}$ be a $P$-ordered compact covering of $X^2\setminus \Delta$. Suppose that $X$ has uncountable tightness. Then, $X$ has a free sequence of length $\omega_1$, hence a convergence free sequence of length $\omega_1$  by \cite{JS1992}. Let $\{x_\alpha: \alpha<\omega_1\}$ be such a sequence and $x^\ast$ the limit point.

Choose $\alpha_\gamma$  and $\beta_\gamma$ in $\omega_1$ for $\gamma\in \omega_1$ such that
 $\alpha_0 < \beta_0 <\alpha_1 <\beta_1< \ldots < \alpha_\gamma <\beta_\gamma <\dots.$ For each $\gamma\in\omega_1$, fix $p_\gamma\in P$ such that $(x_{\alpha_\gamma}, x_{\beta_\gamma})\in K_{p_\gamma}$. Since $P$ has Calibre~$\omega_1$, there is an uncountable subset $\gamma_\tau$ of $\omega_1$ with $\tau<\omega_1$ such that $p_{\gamma_\tau}$ is bounded above by $p^\ast\in P$. Hence,   $K_{p_{\gamma_\tau}}\subseteq K_{p^\ast}$ for each $\tau<\omega_1$, furthermore, $(x_{\alpha_{\gamma_\tau}}, x_{\beta_{\gamma_\tau}})\in K_{p^\ast}$. Since $\{x_\alpha: \alpha<\omega_1\}$ is a convergent sequence with limit point $x^\ast$, the subsequences $x_{\alpha_{\gamma_\tau}}$  and $x_{\beta_{\gamma_\tau}}$ both converge to $x^\ast$, hence, $(x^\ast, x^\ast)\in K_{p^\ast}$ which contradicts with the fact that $K_{p^\ast}\subset X\setminus \Delta$. This contradiction finishes the proof.  \end{proof}

\section{Compact Spaces}\label{compact}

In this section, we study compact spaces possessing a $P$-base with $P$ satisfying some Calibre property. Mainly, we investigate  the problem whether  each compact Hausdorff space with an $\omega^\omega$-base first countable under the assumption $\omega_1<\mathfrak{b}$. We start with some ZFC result about tightness of the spaces with a $P$-base.

 \begin{thm} Let $\kappa$ be an uncountable regular cardinal and  $P$ be a poset with Calibre~$\kappa$.  If $X$ is a compact Hausdorff space with a $P$-base, then $t(X)<\kappa$.

\end{thm}
\begin{proof} Assume, for contradiction, that $t(X)\geq \kappa$. Then, $X$ has a free sequence of length $\kappa$. Hence, by \cite{JS1992}, $X$ has a convergent sequence of length $\kappa$. Let $\{x_\alpha: \alpha<\kappa\}$ be such a sequence and $x^\ast$ be the limit point. Let $S= \{x_\alpha: \alpha<\kappa\}\cup \{x^\ast\}$. Notice that for any unbounded subset $\{\alpha_\gamma: \gamma<\kappa\}$, $x^\ast$ is the limit point of  $\{x_{\alpha_\gamma}: \gamma<\kappa\}$. Fix a neighborhood base $\{B_p: p\in P\}$ at $x_\ast$. It is straightforward to see that $S\setminus B_p$ has size $<\kappa$ for each $p\in P$.

For each $\alpha<\kappa$, choose $p_\alpha \in P$ such that $x_\alpha\notin B_{p_\alpha}$. Let $P'=\{p_\alpha:\alpha\in \kappa\}$. If the cardinality of $P'$ is $<\kappa$, there exists a $p_\alpha\in P'$ such that $S\setminus B_{p_\alpha}$ has size $\kappa$ which is a contradiction. Hence, $P'$ have cardinality $\kappa$. Since $P$ has Calibre~$\kappa$, there is a $\kappa$-sized subset $P''$ of $P'$ which is bounded above. List $P''$ as $\{p_{\alpha_\gamma}: \gamma<\kappa\}$ and pick $p^\ast$ to be the upper bound of $P''$. Then, $S\setminus B_{p^\ast}=\{x_{\alpha_\gamma}: \gamma<\kappa\}$ which is a contradiction. This finishes the proof.    \end{proof}

\begin{cor} Let $P$ be a poset with Calibre~$\omega_1$.  Each compact Hausdorff space with a $P$-base is countable tight.

\end{cor}

A poset $P$ has Calibre~$(\omega_1, \omega)$ if it has Calibre~$\omega_1$. It is showed in \cite{AMETD} that for a separable metric space $M$ the poset $\mathcal{K}(M)$ has Calibre~$\omega_1$ if it has Calibre~$(\omega_1, \omega_1, \omega)$.  Hence it is natural to ask whether a compact space has countable tightness if it has a $P$-base with $P$ having Calibre~$(\omega_1, \omega)$. The following example shows that the answer is negative. So the result above is `optimal' in terms of the Calibre complexity of posets having the form $\mathcal{K}(M)$ with $M$ being a separable metric space.
\begin{ex} There is a poset $P$ with Calibre~$(\omega_1, \omega)$ and a compact space $X$ with a $P$-base, but $t(X)>\omega$.

\end{ex}

\begin{proof} Let $P$ be $\mathcal{K}(\mathbb{Q})$ which clearly has Calibre~$(\omega_1, \omega)$. From Proposition~\ref{o1kq}, the space $\omega_1+1$ has a $P$-base, but its tightness is uncountable. \end{proof}


Again, since $\omega^\omega$ has Calibre~$\omega_1$ under the assumption  $\omega_1<\mathfrak{b}$, we obtain the following result about spaces with an $\omega^\omega$-base.
\begin{cor}\label{ctblt} Assume that $\omega_1<\mathfrak{b}$.  Each compact Hausdorff space with an $\omega^\omega$-base is countable tight.
\end{cor}
It is folklore that any GO-space with countable tightness is first countable. Hence, applying Corollary~\ref{ctblt} to compact GO-spaces, we obtain the following results.
\begin{cor} Let $P$ be a poset with Calibre~$\omega_1$.  Each compact GO-space has a $P$-base if and only if it is first countable.

\end{cor}

The following example shows that the result above doesn't hold for general GO-spaces.

\begin{ex} There is a poset $P$ with Calibre~$\omega_1$ such that there exists a GO-space with a $P$-base and uncountable tightness.   \end{ex}

\begin{proof} Consider the set $\omega_2+1$ in the ordinal order. Let $\mathcal{T}$ be the topology on $\omega_2+1$ such that every point except $\omega_2$ is isolated and a base at $\omega_2$ is $\{(\alpha, \omega_2]: \alpha<\omega_2\}$. So the space $(\omega_2+1, \mathcal{T})$ is a non-first-countable GO-space and clearly has a neighborhood $\omega_2$-base at $\omega_2$. It is straightforward to verify that the poset $\omega_2$ has Calibre~$\omega_1$ since every $\omega_1$-sized subset is bounded above. \end{proof}

The result below was proved in \cite{Banakh2019} through a different approach. We obtain it here as a result of $\omega^\omega$ having Calibre~$\omega_1$ under the assumption $\omega_1<\mathfrak{b}$.

\begin{cor} Assume that $\omega_1<\mathfrak{b}$. Each compact GO-space has an $\omega^\omega$-base if and only if it is first countable. \end{cor}

It is natural to ask as in \cite{Banakh2019} (Problem 8.7.7) whether the same result holds for any compact space. Next, we'll give a positive answer to this problem by showing that any compact space with a $P$-base is first countable if $P=\mathcal{K}(M)$ for some separable metric space $M$ has Calibre~$\omega_1$.

First, we show that any compact space with countable tightness is first countable if it has a $P$-base and $P=\mathcal{K}(M)$ for some separable metric space.  We use the ideas and techniques from \cite{COT2011}.

\begin{thm}\label{ctfstc} Let $P=\mathcal{K}(M)$ for some separable metric space $M$. If $X$ is a compact space with countable tightness
and has  a $P$-base,
then $X$ is first-countable. \end{thm}

\begin{proof} Fix $x\in X$ and an open $P$-base $\{U_p[x]: p\in P\}$ at $x$. For each $p\in P$, let $K_p=X\setminus U_p[x]$. Then, $\{K_p: p\in P\}$ is a $P$-directed compact cover of $X\setminus \{x\}$.

\medskip

For any separable metric space $M$, the space $P=\mathcal{K}(M)$ with the Hausdorff metric $d^H$ is also separable, hence second countable. Also if $\{p_n: n\in\omega\}$ is a sequence  converging to $p$ in $P$, then $p^\ast=p\cup (\bigcup\{p_n: n\in\omega\})$ is compact, hence it is an element in $P$ with $p_n\subset p^\ast$ and $p\subset p^\ast$.

Fix a countable base $\{B_n: n\in\omega\}$ of $P$. For each $n\in\omega$, define $L(B_n)=\bigcup \{K_p: p\in B_n\}$. And for each $p\in P$, define $C(p)=\bigcap\{L(B_n): p\in B_n\}$. For each $p\in P$, we  pick  a decreasing local base $\{B_{n_i}^p: i\in \omega\}\subseteq \{B_n: n\in\omega\}$ at $p$ such that for each $i\in \omega$ there is a positive number $\epsilon_i$ satisfying that $B_{n_i}^p\supset D_{d^H}(p, \epsilon_i)\supset\overline{B_{n_{i+1}}^p}$  where $D_{d^H}(p, \epsilon_i)$ is the open ball centered at $p$ with radius $\epsilon_i$. Define  $C'(p)=\bigcap\{L(B_{n_i}^{p}): i\in \omega\}$. It is straightforward to verify that $C'(p)= C(p)$. 

\medskip

First, we claim that $x$ is not in the closure of $C(p)$ for all $p\in P$. Fix $p\in P$. By the countable tightness of $X$, it suffices to show that if $x$ is not in the
 closure of any countable subset of $C(p)$. Let $\{ y_i : i \in\omega\}$ be a  countable subset of $C(p)$. For each $i\in\omega$, choose $q_i\in B_{n_i}^p$ with $y_i\in K_{q_i}$. Clearly $\{q_i: i\in \omega\}$ is a sequence converging to $p$, hence $p^\ast =p\cup (\bigcup \{q_i: i\in \omega\})$ is an element in $P$ with
 $\{y_i:i\in\omega\}\subset  K_{p^\ast}$ which implies that $x$ is not in the closure of $\{ y_i : i\in\omega\}$.

 \medskip
 Then, we claim that for each $p\in P$, there is an $i\in \omega$ such that $x$ is not in the closure of $L(B_{n_i}^p)$.  Fix $p\in P$. Choose any open set $U$ such that $\overline{C(p)}\subset U$ and $x\notin \overline{U}$. It suffices to prove that there is an $i$ so that  $L(B_{n_i}^p)\subset U$.  Suppose not. Choose $y_i\in L(B_{n_i}^p)\setminus U$ for each $i\in \omega$. Then for each $i\in\omega$, choose $q_i\in B_{n_i}^p$ so that $y_i\in K_{q_i}$. Define $p_i^\ast =p\cup (\bigcup \{q_j: j> i\})$. Hence $\{y_j: j>i\}\subset K_{p_i^\ast}$. By the property of the Hausdorff metric $d^H$, it is straightforward to verify that  $d^H(p_{i}^\ast, p)\leq \epsilon_i$, hence $p_{i}^\ast\in B_{n_i}^p$ which implies that $K_{p_{i}^\ast}\subset L(B_{n_i}^p)$ for each $i\in\omega$. Therefore, $\bigcap \{K_{p_i^\ast}:i\in \omega\}\subseteq C(p)$. Then, all the limit points of $\{y_i: i\in \omega\}$ are in $C(p)$ which contradicts with $C(p)\subset U$ and the choices of $\{y_i: i\in \omega\}$.

\medskip

Finally, we prove that the family $\mathcal{L}=\{\overline{L(B_n)}: x\notin \overline{L(B_n)}\}$ is a cover of $X\setminus \{x\}$, furthermore, $\{B: B=X\setminus S \text{ for some }S\in \mathcal{L}\}$ is a local base at $x$. For each $p\in P$, there is an $i\in \omega$ such that that $x$ is not in the closure of $L(B_{n_i}^p)$. Hence $\overline{L(B_{n_i}^p)}\in \mathcal{L}$. Since $K_p\subset L(B_{n_i}^p)$, this completes the proof.  \end{proof}

\begin{thm} Let $P=\mathcal{K}(M)$ for some separable metric space $M$ such that $P$ has Calibre~$\omega_1$. Any compact space $X$ with a $P$-base is first-countable. \end{thm}

\begin{proof}By Corollary~\ref{ctblt}, $X$ has countable tightness since $P$ has Calibre~$\omega_1$. Then by Theorem~\ref{ctfstc}, $X$ is first-countable.  \end{proof}

Then using the fact that $\omega^\omega$ has Calibre~$\omega_1$ under the assumption $\omega_1<\mathfrak{b}$, we get a positive answer to Problem 8.7.7 in \cite{Banakh2019}.

\begin{cor} Assume $\omega_1<\mathfrak{b}$. A compact space has an $\omega^\omega$-base if and only if it is first countable.
 \end{cor}

 \section{Scattered Compact Spaces}\label{scacompact}

We recall that a topological space $X$ is scattered if each non-empty subspace of $X$ has an isolated point. The complexity of a scattered space can be determined by the scattered height.

For any subspace $A$ of a space $X$, let $A'$ be the set of all non-isolated
points of $A$. It is straightforward to see that $A'$ is a closed subset of $A$. Let $X^{(0)} = X$ and define $X^{(\alpha)} = \bigcap_{\beta<\alpha} (X^{(\beta)} ) '$ for each $\alpha > 0$. Then a space $X$ is  scattered if $X (\alpha) = \emptyset$ for some ordinal $\alpha$. If $X$ is scattered, there exists a unique ordinal $h(x)$ such that $x\in X^{(h(x))}\setminus X^{(h(x)+1)}$ for each $x\in X$. The ordinal $h(X)=\sup\{h(x): x\in X\}$  is called the scattered
height of $X$ and is denoted by $h(X)$. It is known that any compact scattered space is zero-dimensional. Also, it is straightforward to show that for any compact scattered space $X$, $X^{(h(x))}$ is a non-empty finite subset.

\begin{thm}\label{scomg1} Let $P$ be a poset with Calibre~$\omega_1$ and $X$ a scattered compact space with a $P$-base. Then $X$ is countable. \end{thm}
\begin{proof} If $h(X)=0$, then $X$ is countable because it is compact.

Assume $h(X)=\alpha$ and any compact scattered space with a $P$-base is countable if it has a scattered height $<\alpha$.  Since $X$ is compact, $X^{(\alpha)}$ is a nonempty finite subset of $X$. List $X^{(\alpha)}=\{x_1, \ldots, x_n\}$. For each $i\in \{1, \ldots, n\}$, take a closed and open neighborhood $U_i$ of $x_i$ with $U_i \cap X(\alpha)=\{x_i\}$. Then $X\setminus \bigcup \{U_i: i=1, \ldots, n\}$ is a scattered compact space with scattered height $<\alpha$, hence it is countable by the assumption. So it is sufficient to show that $U_i$ is countable for each $i=1, \ldots, n$.

Fix $i\in\{1, \ldots, n\}$. Consider the subspace $Y=U_i\cap X$. By proposition~\ref{sub},  $Y$ has a neighborhood $P$-base $\{B_p: p\in P\}$ at $\{x_i\}$.  For each $p\in P$, $Y\setminus B_p$ is a compact subspace with scattered height $<\alpha$, hence is countable by the inductive assumption.

Assume that $Y$ is uncountable. Take an uncountable subset $\{y_\alpha: \alpha<\omega_1\}$ of $Y\setminus \{x_i\}$.  For each $\alpha<\omega_1$, we choose $p_\alpha\in P$ such that $y_\alpha\notin B_{p_\alpha}$.

If $\{p_\alpha: \alpha<\omega_1\}$ is countable, there is a $p^\ast\in \{p_\alpha: \alpha<\omega_1\}$ such that there is an uncountable subset $D$ of $\{y_\alpha: \alpha<\omega_1\}$ such that $D\subset Y\setminus B_{p^\ast}$ which is a contradiction.

If $\{p_\alpha: \alpha<\omega_1\}$ is uncountable, then it has an uncountable subset $P'$ which is bounded above using the Calibre~$\omega_1$ property of $P$. List $P'=\{p_{\alpha_\gamma}: \gamma<\omega_1\}$. Let $p^\ast$ be an upper bound of $P'$. Then we have that $y_{\alpha_\gamma}\notin B_{p^{\ast}}$ for each $\gamma<\omega_1$. This also contradicts with the fact that $Y\setminus B_{p^{\ast}}$ is countable. This finishes the proof. \end{proof}

Using the same approach we obtain the following example.
\begin{ex} The one point Lindel\"{o}fication of uncountably many points doesn't have a $P$-base if $P$ has Calibre~$\omega_1$, hence, under the assumption $\omega_1<\mathfrak{b}$, it doesn't have an $\omega^\omega$-base.
\end{ex}

The example above uses the fact that under the assumption $\omega_1<\mathfrak{b}$, the poset $\omega^\omega$ has Calibre~$\omega_1$. Furthermore, using Theorem~\ref{scomg1}, we obtain the following result which  answers Problem 8.6.9 in \cite{Banakh2019} positively. This also gives a partial positive answer to Problem 8.6.8 in the same paper.

\begin{cor} Assume $\omega_1<\mathfrak{b}$. Any scattered compact space with an $\omega^\omega$-base is countable, hence metrizable.  \end{cor}


It was proved  in \cite{Banakh2019} that any compact spaced with an $\omega^\omega$-base and finite scattered height is countable, hence metrizable. Next, we show that the same result holds for any compact space with a $P$-base and finite scattered height if $P$ has Calibre~$(\omega_1, \omega)$.

\begin{thm} Let $P$ be a poset with Calibre~$(\omega_1, \omega)$ and $X$  a compact Hausdorff space with a $P$-base and finite scattered height. Then $X$ is countable, hence metrizable. \end{thm}

\begin{proof} Fix a natural number $n>0$. Assume that any compact Hausdorff space with scattered height $<n$ is countable. Let $X$ be a compact Hausdorff space with scattered height $n$. We'll show that $X$ is countable.  As discussed in the proof of Theorem~\ref{scomg1}, we could assume that $X^{(n)}$ is a singleton, denoted by $x$, without loss of generality. Suppose, for contradiction, that $X$ is uncountable. Define $m$ to be the greatest natural number such that $X^{(m)}$ is uncountable and $X^{(m+1)}$ is countable. Then there are two cases: 1. $m=n-1$; 2. $m<n-1$. We will obtain contradictions in both cases.

First assume that $m=n-1$. Then we fix a neighborhood $P$-base $\{B_p: p\in P\}$ at $x$. For each $p\in P$, $X^{(m)}\setminus B_p$ is finite as $X$ is compact and $B_p$ is open. Pick an uncountable subset $\{x_\alpha: \alpha<\omega_1\}$ of $X^{(m)}$ with $x_\alpha\neq x$ for all $\alpha<\omega_1$. For each $\alpha<\omega_1$, there is a $p_\alpha\in P$ such that $x_\alpha\notin B_{p_\alpha}$. If $\{p_\alpha: \alpha<\omega_1\}$ is countable, then there exists $p^\ast\in \{p_\alpha: \alpha<\omega_1\}$ such that $X^{(m)}\setminus B_{p^\ast}$ is uncountable which is a contradiction.  If $\{p_\alpha: \alpha<\omega_1\}$ is uncountable,  we can find a countable bounded subset $\{p_{\alpha_n}: n\in\omega\}$ of $\{p_\alpha: \alpha<\omega_1\}$ using the Calibre~$(\omega_1, \omega)$ property of $P$. Let the upper bound of $\{p_{\alpha_n}: n\in\omega\}$ be $p^\ast$. Then, $x_{\alpha_n}\notin B_{p^\ast}$ for each $n\in\omega$. This is a contradiction.

Now we assume that $m<n-1$. Then $X^{(m+1)}\setminus \{x\}$ is countable which can be  listed as $\{x_\ell: \ell\in \omega\}$. For each $\ell$, pick a closed and open neighborhood $U_\ell$ of $x_\ell$. Then for each $\ell<\omega$, $U_\ell$ is a compact subspace with scattered height $<n$, hence is countable. Therefore, $X^{(m)}\setminus \bigcup \{U_\ell:\ell\in \omega\}$ is uncountable. Pick an uncountable subset $S=\{x_\alpha: \alpha<\omega_1\}$ of $X^{(m)}\setminus (\{x\}\cup (\bigcup \{U_\ell:\ell\in \omega\}))$. Fix a neighborhood $P$ base $\{B_p: p\in P\}$ at $x$. For each $p\in P$, $S\setminus B_p$ is finite. Similarly as in the proof of case 1, we could obtain a contradiction using the Calibre~$(\omega_1, \omega)$ property of $P$.  \end{proof}

The result above doesn't hold for compact space with uncountable scattered height since the space $\omega_1+1$ has a $\mathcal{K}(\mathbb{Q})$-base. However, we don't know the answer to the following problem.

\begin{qu} Assume that $\omega_1=\mathfrak{b}$. Let $P$ be a poset with Calibre~$(\omega_1, \omega)$ and $X$ be any compact Hausdorff space with a $P$-base and countable scattered height. Is $X$ countable? \end{qu}





\section{Calibre $\omega_1$ and  non-first-countable compact space}\label{calbreo1}

We prove that there is a model of Martin's Axiom in which  there is a
compact space that has a $P$-base  for a  poset
 $P$ with Calibre $\omega_1$. This space will be the space constructed
by Juhasz, Koszmider, and Soukup in the paper
\cite{JKS2009}. This article \cite{JKS2009}
shows there is a forcing notion
that forces the existence of a
 first-countable initially $\omega_1$-compact locally compact space
 of cardinality $\omega_2$
 whose one-point compactification has countable tightness. We must
 prove that there is a poset $P$ as above. We must also show that
 extra properties of the space ensure that we can perform a further
 forcing to obtain a model of
 Martin's Axiom and that the desired properties of  a space naturally
 generated from the original space possesses these same properties in
 the final model.  The reader may be interested to note that in this
 way we produce a model of Martin's Axiom and $\mathfrak c=\omega_2$
in which there is a compact
 space of countable tightness that is not sequential. This is
 interesting because Balogh proved in \cite{Balogh1989}
that the
forcing axiom, PFA, implies that compact spaces of countable tightess
are sequential. It was first shown in \cite{Dow2016} that the
celebrated Moore-Mrowka problem was independent of Martin's Axiom
plu $\mathfrak c=\omega_2$. The methods in \cite{Dow2016} are indeed
based on the paper \cite{JKS2009} using the notion of T-algebras first
formulated in \cite{Kosz1999}. The example in \cite{JKS2009} is itself
a space generated by
a T-algebra but is not explicitly
formulated as such because of its simpler structure.

 To do all this, at minimum cost, we must explicitly reference a
 number of statements and proofs from \cite{JKS2009}.
The construction is modeled on the following
natural property of locally compact scattered topology,
 $\tau$, with  base set an ordinal $\mu$ in which
 initial segments are open. The well-ordering on
 the underlying set arises canonically from the fact
 that such spaces are right-separated and scattered.
    There are functions
 $H$ with domain $\mu$ and a function $i:[\mu]^2 \rightarrow
[\mu]^{<\aleph_0}$ satisfying that for all
 $\alpha<\beta <\mu$:
\begin{enumerate}
\item $\alpha \in H(\alpha)\subset\alpha+1$ and
  $H(\alpha)$ is a compact open set (i.e. $H(\alpha)\in\tau$),
  \item $i(\alpha,\beta)$ is a finite subset of $\alpha$,
  \item if $\alpha \notin H(\beta)$, then
    $H(\alpha)\cap H(\beta)\subset \bigcup
    \{ H(\xi) : \xi \in i(\alpha,\beta)\}$
  \item if $\alpha\in H(\beta)$, then
        $H(\alpha)\setminus  H(\beta)\subset \bigcup     \{ H(\xi) :
    \xi \in i(\alpha,\beta)\}$.
\end{enumerate}
Conversely if $H$ and $i$ are functions as in (1)-(4) where
(1) is replaced by simply
\medskip

(1') $\alpha \in H(\alpha)\subset\alpha+1$ (i.e. no mention of
topology)
\medskip

\noindent then using the family $\{ H(\alpha) : \alpha \in \mu\}$ as
a clopen subbase generates a locally compact scattered topology
on $\mu$ in which $H, i$ satisfy property (1)-(4).

Statements (3) and (4) are combined into a single statement in
\cite{JKS2009} by adopting the notation
$$H(\alpha)*H(\beta) = \begin{cases}
  H(\alpha)\cap H(\beta) & \alpha\notin H(\beta)\\
  H(\alpha)\setminus H(\beta) & \alpha\in H(\beta)
\end{cases}~.$$
As noted in \cite{JKS2009}  a locally compact scattered space
can not have the properties listed above, hence the construction must
be generalized.  Also it is shown above (and in
 \cite{Banakh2019} for $P=\omega^\omega$)
  that a compact scattered space with a $P$-base that
 has Calibre $\omega_1$ will be first countable. \bigskip

 The generalization from \cite{JKS2009} will use almost the same
 terminology and ideas to generate a topology on the base set
 $\omega_2\times \mathbb{C}$ where $\mathbb{C}=2^{\mathbb N}$ is the
 usual Cantor set and, for each $\alpha<\omega_2$,
 $\{\alpha\}\times\mathbb{C}$ will be homeomorphic to $\mathbb{
   C}$.  For $n\in \mathbb{N}=\omega\setminus\{0\}$
 and $\epsilon\in 2$, the notation
 $[n,\epsilon]$ will denote the clopen subset
  $\{ f\in 2^{\mathbb{N}} : f(n)=\epsilon\}$ in $\mathbb{C}$.
 However a critically important aspect of the construction to
 watch for is that every point of the space will have a local base of
 neighborhoods that \textit{splits} only one of the sets in
  $\{ \{\alpha\}\times \mathbb{C} : \alpha\in \omega_2\}$. The
 function $H$ will identify the copies that such a subbasic clopen set
 meets (and contains all except the \textit{top} one).
 More precisely,  $H(\alpha,0)\times \mathbb{C}$
be a subbasic clopen set, and for $n>0$,
the set
 $H(\alpha,n)\subset H(\alpha,0)$ will be used to construct the
subbasic clopen set that intersects $\{\alpha\}\times \mathbb{C}$ as
the set $\{\alpha\}\times[n,1]$. Naturally, $H(\alpha,0)\setminus
H(\alpha,n)$ will generate the subbasic clopen set
corresponding to $\{\alpha\}\times [n,0]$. The function $i$ is
similarly
generalized to be a function from $[\omega_2]^2\times\omega$
into $[\omega_2]^{<\aleph_0}$.  Here is
the definition of a suitable pair of functions from
\cite{JKS2009}*{Definition 2.4}.

\begin{definition} A pair $H:\omega_2\times\omega \rightarrow
   \mathcal P(\omega_2)$ and $i:[\omega_2]^{2}\times\omega\rightarrow [\omega_2]^{<\aleph_0}$ is
$\omega_2$-suitable if the\label{suitable} following conditions hold
   for all $\alpha <\beta <\omega_2$ and $n\in\mathbb{N}$:
   \begin{enumerate}
   \item $\alpha \in H(\alpha,n)\subset H(\alpha,0)\subset \alpha+1$,
   \item $i(\alpha,\beta,n)\in[\alpha]^{<\aleph_0}$,
   \item $H(\alpha,0)*H(\beta,n)\subset
     \bigcup\{ H(\xi,0) : \xi\in i(\alpha,\beta,n)\}$.
   \end{enumerate}
   Also, given an $\omega_2$-suitable pair $H,i$,  define the following sets
   for $\alpha\in \omega_2$, $F\in [\omega_2]^{<\aleph_0}$
 and $n\in \mathbb{N}$:
 \begin{enumerate}
   \setcounter{enumi}{3}
     \item $U(\alpha) = U(\alpha,\mathbb{C}) =
       H(\alpha,0)\times\mathbb{C}$,
     \item $U(\alpha,[n,1]) = (\{\alpha\}\times [n,1])
       \ \cup \ \left( (H(\alpha,n)\setminus \{\alpha\}) \times
       \mathbb{C}\right)$,
     \item $U(\alpha,[n,0]) = U(\alpha,\mathbb{C})\setminus
       U(\alpha,[n,1])$,
       \item $U[F] =\bigcup \{ U(\xi) : \xi\in F\}$.
    \end{enumerate}
\end{definition}

Next we rephrase \cite{JKS2009}*{Lemma 2.5}:

\begin{proposition}
  If $H, i$ is an $\omega_2$-suitable pair then\label{localbase}
    the subbase
  $$\{ U(\alpha,\mathbb{C}) :\alpha\in\omega_2\}
   \cup \{U(\alpha,[n,\epsilon]) : \alpha\in \omega_2, n\in \mathbb N,
   \epsilon \in 2\}$$
   generates a locally compact Hausdorff topology $\tau_H$
 on $\omega_2\times
   \mathbb{C}$ satisfying that
  for all
   $\alpha\in\omega_2$, $n\in\mathbb{N}$, and $r\in \mathbb{C}$,
   \begin{enumerate}
   \item $U(\alpha,\mathbb{C})$,  $U(\alpha,[n,1])$ are compact,
     \item the collection of finite intersections of members of the family
     $$\{ U(\alpha, [n,r(n)]) \setminus U[F]
       : n\in\mathbb N, F\in [\alpha]^{<\aleph_0}\}$$  is a local base at
        $(\alpha,r)$
   \end{enumerate}
\end{proposition}

Next, the authors of \cite{JKS2009} have to work very hard to produce
an $\omega_2$-suitable pair so that $\tau_H$ is first-countable and
initially $\omega_1$-compact. The first step is to work in a model in
which there is a special function $f:[\omega_2]^2 \mapsto
[\omega_2]^{\leq\aleph_0}$ called a strong $\Delta$-function. Since we will not
need any properties of this function we omit the definition, but
henceforth assume that $f$ is such a function.  We
record additional minor modifications of
results from \cite{JKS2009}*{4.1,4.2}.

\begin{proposition} There is a ccc poset $P_f$ consisting\label{giveH} of
  quadruples $p = (a_p,h_p,i_p,n_p)$ that are finite approximations of
  an $\omega_2$-suitable pair where
\begin{enumerate}
\item $a_p\in[\omega_2]^{<\aleph_0}$, $n_p\in \omega$
\item    $h_p : [a_p]^2\times n_p \mapsto \mathcal P(a_p)$,
\item $i_p : [a_p]^2\times n_p\mapsto [a_p]^{<\aleph_0}$,
\end{enumerate}
and, for each $P_f$-generic filter $G$, the relations
$$ H = \bigcup \{ h_p : p\in G\}\ \ \mbox{and}\ \ \ i =
\bigcup \{ i_p : p\in G\}$$
are functions that form  an $\omega_2$-suitable pair.
In particular,
 if $p\in G$, $\alpha\in h_p(\beta)$,
 and $i_p(\alpha,\beta,0)=\emptyset$, then
 (in $V[G]$)
 $U(\alpha,\mathbb{C})\subset
 U(\beta,\mathbb{C})$.
\end{proposition}

The space $(\omega_2\times \mathbb{C},\tau_H)$ is shown to have these
additional properties \cite{JKS2009}*{4.2}:

\begin{proposition} If $G$ is $P_f$-generic and $H,i$ are
 defined\label{mainProp}
  as in Proposition \ref{giveH}, then the following hold in $V[G]$:
  \begin{enumerate}
\item $X_H = (\omega_2\times \mathbb{C}, \tau_H)$ is locally compact
  0-dimensional of cardinality $\mathfrak c=2^{\aleph_1}=\aleph_2$,
\item $X_H$ is first-countable,
  \item for every $A\in [X_H]^{\omega_1}$, there is a $\lambda
    <\omega_2$ such that $A\cap U(\lambda,\mathbb{C})$ is uncountable,
    \item for every countable $A\subset X_H$, either $\overline{Y}$ is
      compact or there is an $\alpha<\omega_2$ such that
       $(\omega_2\setminus \alpha)\times \mathbb{C}$ is contained in
      $\overline{Y}$.
  \end{enumerate}
  Consequently $X_H$ is a locally compact, $0$-dimensional, normal,
  first-countable, initially $\omega_1$-compact but non-compact space.
\end{proposition}

Finally, we need the following strengthening
of \cite{JKS2009}*{Lemma 7.1} but which is actually proven.

\begin{proposition} If $p= (a_p,h_p,i_p,n_p)\in P_f$ and
  $a_p\subset\lambda\in \omega_2$, then there is a $q < p$ in $P_f$
  such\label{strongextend},
  that
  \begin{enumerate}
\item  $a_q = a_p\cup \{\lambda\}$ and $n_q=n_p$,
\item  $a_p\subset h_q(\lambda,0)$,
\item $i_q(\alpha,\lambda,j)=\emptyset$ for all $\alpha\in a_p$ and
  $j<n_q$.
  \end{enumerate}
\end{proposition}

We note that for $p,q$ as in Lemma \ref{strongextend},
if $q$ is in the generic filter $G$, then
 $U(\alpha, \mathbb{C})$ is a subset of $U(\lambda,\mathbb{C})$ for
all
$\alpha\in a_p$.  One consequence of this is that the family
 $\{ U(\alpha,\mathbb{C}) : \alpha\in \omega_2\}$ is finitely upwards
directed. Equivalently, the family of complements of these sets in the
one-point compactification of $X_H$ is a neighborhood base for the
point at infinity.

Now we strengthen  \cite{JKS2009}*{Lemma 7.2} which will be used to
prove that the one-point compactification of $X_H$ has Calibre
$\omega_1$.
Some of our proofs will
  require forcing arguments and we refer the reader to \cite{Kunen1980}
  for more details. However some remarks may be sufficient
  to assist many readers.
The forcing extension,
  $V[G_Q]$ by a  $Q$-generic filter $G_Q$ for a poset $Q$ is
equal to the valuation, $\val_{G_Q}(\dot A)$
 for  the collection of all $Q$-names $\dot A$ that are
 sets from $V$.  The notation $q\Vdash x\in \dot A$ can be read
 as the assertion that $x\in \val_{G_Q}(\dot A)$ for any
 generic filter with $q\in G_Q$.
 The forcing theorem (\cite{Kunen1980}*{VII 3.6})
 ensures, for example, that if $\dot A$ is a $Q$-name of a subset of a ground
 model set $B$, then $b$ is an element of $\val_{G_Q}(\dot A)$ exactly
 when there is an element $q\in G_Q$ such that
  $q\Vdash b\in \val_{G_Q}(\dot A)$.
 Additionally,
 the set of $q\in Q$ that satisfy that $q\Vdash x\in \dot A$ is a set
 in the ground model, as is the set of $x$ for which there exists
 a $q$ with $q\Vdash x\in \dot A$.  This justifies the first line
 of the next proof.

\begin{lemma} In $V[G]$, for each uncountable $A\subset\omega_2$,
  there is  a $\lambda<\omega_2$ such that $U(\alpha,\mathbb{C})
  \subset U(\lambda,\mathbb{C})$ for uncountably many $\alpha\in A$.
\end{lemma}

\begin{proof}
   Let $\dot A$ be a
  $P_f$-name for a subset of $\omega_2$.
  Fix any condition $p\in G$ and assume
  that $p$ forces that $\dot A$ has cardinality $\aleph_1$.
  We prove that there is a $q<p$ and a $\lambda\in a_q$
satisfying that if $q\in G$
then  there are uncountably many $\alpha\in \val_G(\dot A)$
such that $U(\alpha,\mathbb{C})\subset
U(\lambda,\mathbb{C})$. It is a standard fact of forcing that this
would then establish the Lemma (i.e. that there is then necessarily
such a $q\in G$).

Let $I$ denote the set of $\alpha\in \omega_2$ satisfying that there
is some $p_\alpha < p$ (which we choose) forcing that
$\alpha\in \dot
A$. Since $p$ forces that $\dot A$ is a subset of $I$ it follows
that $I$ has cardinality at least $\omega_1$. Since $P_f$ is ccc, it
also follows that $I$ has cardinality equal to $\omega_1$ but it
suffices for this argument
to choose any $\lambda\in \omega_2$ such that $I\cap \lambda$ is
uncountable. For each $\alpha\in I$, choose $q_\alpha < p_\alpha$ so
that $a_{q_\alpha}=a_p\cup \{\lambda\}$ and the properties of the
pair $p_\alpha,q_\alpha$ are as stated in Proposition \ref{strongextend}.

Just as in the proof of \cite{JKS2009}*{Lemma 7.2}, the fact that
$P_f$ is ccc ensures that there is some $q<p$ such that
so long as $q\in G$, the set $\{ \alpha \in I\cap \lambda
: q_\alpha\in G\}$ is uncountable. As remarked after Proposition
\ref{strongextend}, it  follows, in $V[G]$,
 that
 $U(\alpha,\mathbb{C})\subset U(\lambda,\mathbb{C})$
 for all $\alpha \in \{\alpha\in I \cap \lambda : q_\alpha\in G\}$.
\end{proof}

\begin{theorem} If $G$ is a $P_f$-generic filter, then in
  $V[G]$, the one-point compactification of the space $X_H$
  has a $P$-base for a poset with Calibre $\omega_1$.
\end{theorem}

\begin{proof}
  The poset $P$ consists of the family $\{ U(\alpha,\mathbb{C}) :
  \alpha\in \omega_2\}$ ordered by inclusion. To complete
  the proof we have to note that $\omega <_T P$. For this it is enough
  to prove that there is a countable subset of $P$ that has no upper
  bound. It is relatively easy to prove that $X_H$ is separable
  (indeed, that $\omega\times \mathbb{C}$ is dense) but oddly enough
  this is not stated in \cite{JKS2009} and we can more easily
  simply note that $X_H$ is not $\sigma$-compact because
  by Proposition \ref{mainProp}  it
 is countably compact and non-compact.
\end{proof}

An important feature of the construction of $X_H$ from the
$\omega_2$-suitable pair $H,i$ is that even in a forcing
extension by a ccc poset $Q$ (in fact by
any poset that preserves
that $\omega_1$ and $\omega_2$ are cardinals), the new
interpretation of the space obtained using
$H,i$ (i.e. the base set $\omega_2\times \mathbb{C}$ may change
because
there can be new elements of $\mathbb{C}$) is still
locally compact and $0$-dimensional. This is similar to the fact
that local compactness of scattered spaces is preserved by any forcing
(a result by Kunen).
The other properties
of $X_H$, such as first-countability and initial
$\omega_1$-compactness, as well as properties
of its one-point extension are not immediate and
will depend on what subsets of $\omega_2$ have been added.
\bigskip

An unexpected feature of the $\omega_2$-suitable pair is that, in
fact,
the first countability of $X_H$ is preserved by any forcing.

\begin{lemma} For each poset $Q$ in $V[G]$ and\label{character}
 $Q$-generic filter
  $G_Q$, the space $X_H$ is first-countable in $V[G][G_Q]$.
\end{lemma}

\begin{proof}
Of course we will use the fact that,
in $V[G]$, $X_H$ is first-countable (as stated in Proposition \ref{mainProp}).
Fix any $\alpha\in \omega_2$ and recall from Proposition \ref{localbase},
that the collection of all finite intersections of the family
$$\{ U(\alpha, [n,r(n)]) \setminus U[F] : n\in \mathbb{N}, \ F\in [\alpha]^{<
\aleph_0} \}$$
is a local base at $(\alpha,r)\in \{\alpha\}\times \mathbb{C}$
 (in any model).
 In $V[G]$, for each $r\in \mathbb{C}$, let
$Z(\alpha,r) = \bigcap_{n\in\mathbb{N}}U(\alpha,[n,r(n)])$
and let $K(\alpha,r) =
\{ \xi< \alpha :
 \{\xi\}\times \mathbb{C}\subset Z(\alpha,r) \}$. Let us
 recall that $\xi\in K(\alpha,r)$ if and only
  if $Z(\alpha,r)\cap (\{\xi\}\times\mathbb{C})$ is not empty. Similarly,
  by the definition of $U(\alpha,[n,r(n)])$ given in Definition \ref{suitable},
     $K(\alpha,r) = \bigcap \{ H(\alpha,[n,r(n)]) : n\in \mathbb N\}$.
    Since, for all $n\in \mathbb{N}$,
      $\{H(\alpha,[n,0]), H(\alpha,[n,1])\}$
is a partition of $H(\alpha,0)$,       it follows that
      $\{ K(\alpha, r) : r\in \mathbb{C}\}$ is also a partition of
       $H(\alpha,0)$.  Since $X_H$ is first-countable (in $V[G]$),
       for each $r\in \mathbb{C}$, there is a
       countable $F_r\subset K(\alpha,r)$ such that
       $K(\alpha,r)\times\mathbb{C}\subset \bigcup\{ U(\xi,0) : \xi
       \in F_r\}$.

       Now we are ready to show that, in $V[G][G_Q]$,
       each point of  $\{\alpha\}\times\mathbb{C}$ is a $G_\delta$-point
        in
        $X_H$.  For each $r\in \mathbb{C}$, we again define the $G_\delta$-set
         $Z(\alpha,r)$ and
         $K(\alpha,r)\subset H(\alpha,0)$ as we did in $V[G]$ but
         as \textit{calculated\/} in the new model $V[G][G_Q]$.
         It is immediate that $Z(\alpha,r)\cap (\{\alpha\}\times\mathbb{C})$
         is equal to $(\alpha,r)$.
         Since there are no changes to the values of
          $H(\alpha,[n,\epsilon])$ for $(n,\epsilon)\in \mathbb{N}\times 2$,
           the value of $K(\alpha,r)$ for each $r\in \mathbb{C}\cap V[G]$ is
           unchanged  and the family
            $\{ K(\alpha,r) : r\in \mathbb{C}\}$ is a partition of $H(\alpha,0)$.
            It clearly remains the case that,
            for $r\in \mathbb{C}\cap V[G]$, $K(\alpha,r)\times
            \mathbb{C}$ is a subset of $\bigcup\{U(\xi,0): \xi\in F_r\}$.
            This implies that $(\alpha,r)$ is a $G_\delta$-point for
            each $r\in \mathbb{C}\cap V[G]$. Now consider a point
             $s\in \mathbb{C}$ that is not an element of $V[G]$.
But now we have that $K(\alpha,s)$ is empty since
  $H(\alpha,0)$ is covered by the family $\{ K(\alpha, r) :
     r\in \mathbb{C}\cap V[G]\}$. This implies that
       $Z_s$ is equal to the singleton set $\{ (\alpha,s)\}$.
\end{proof}

Next we prove that we can extend the model $V[G]$ to obtain a model in
which Martin's Axiom holds (and $\mathfrak c = \omega_2$). We do so
using the following result from \cite{Kunen1980}*{VI 7.1, VIII 6.3}
(i.e. the standard
method to construct a model of Martin's Axiom).

\begin{proposition} In the model $V[G]$, there is an increasing chain
   $\{ Q_\xi : \xi\leq\omega_2\}$ of partially ordered sets\label{MA}
  satisfying for each $\xi<\omega_2$
  \begin{enumerate}
  \item  $Q_\xi$ is a ccc poset of cardinality at most $\aleph_1$,
    \item each maximal antichain of $Q_\xi$ is a maximal antichain of
      $Q_{\omega_2}$,
    \item if $G_2$ is a $Q_{\omega_2}$-generic filter, then
      in the model $V[G][G_2]$
      \begin{enumerate}
      \item     Martin's Axiom holds and $\mathfrak c=\omega_2$
        \item for each $A\subset \omega_2\times\mathbb{C}$ of
          cardinality less than $\omega_2$, there is a $\xi<\omega_2$
          such that $A$ is in the model $V[G][G_2\cap Q_\xi]$.
      \end{enumerate}
        \end{enumerate}
\end{proposition}

 For the remainder of this section let $\{ Q_\xi : \xi\leq\omega_2\}$
 be the poset as in this Proposition and let $G_2$ be a
  $Q_{\omega_2}$-generic filter. The model
   $V[G][G_2\cap Q_\xi]$ is actually equal to the valuation
   by
    $G_2 $ of all $Q_\xi$-names that are in $V[G]$.

First we prove that the poset of $P$ (consisting of the family
 $\{ U(\alpha, \mathbb{C}) : \alpha\in\omega_2\}$ ordered by
inclusion) still has Calibre $\omega_1$ in the forcing extension
of $V[G]$ by $Q_{\omega_2}$. In fact, by Proposition \ref{MA}, it
suffices to prove that any ccc poset $Q$ of cardinality at most
$\aleph_1$ preserves that $P$ has Calibre $\omega_1$.

\begin{lemma} If $G_Q$ is $Q$-generic over $V[G]$
for a  ccc poset, then  $P$ has Calibre $\omega_1$ in the model
$V[G][G_Q]$.
\end{lemma}

\begin{proof} Let $\dot A$ be a $Q$-name of a subset
  of $\omega_2$ and let $q$ be any element of $Q$.
Let $I$ be the set of $\alpha\in \omega_2$ such that there exists some
$q_\alpha < q$ such that $q_\alpha\Vdash \alpha\in \dot A$. For
each $\alpha\in I$ choose such a $q_\alpha < q$. Fix any
$\lambda<\omega_2$ so that $I_\lambda = \{ \alpha \in  I
: U(\alpha,\mathbb{C})\subset U(\lambda,\mathbb{C})\}$ is
uncountable. Choose $\bar q<q$ so that
for all $Q$-generic $G_Q$ with $\bar q\in G_Q$,
the set $\{ \alpha \in I_\lambda : q_\alpha\in G_Q\}$ is uncountable.
Since $\alpha\in\val_{G_Q}(\dot A)$ for all $\alpha\in I$
with $q_\alpha\in G_Q$, this completes the proof
that $P$ retains the Calibre $\omega_1$ property.
\end{proof}

It follows from the results so far that, in the model
 $V[G][G_2]$,  the one-point compactification of $X_H$
 has a $P$-base and that $P$ has Calibre $\omega_1$. Also,
   $\{ U(\alpha,0) : \alpha\in \omega_2\}$ is an open cover
   of $X_H$ that has no countable subcover, so the one-point
   compactification is not first-countable. This completes the proof of
   the desired properties,   but it is of independent interest
   to prove this next result because of the connection to the Moore-Mrowka problem.

   \begin{theorem} In the model $V[G][G_2]$
   there is an $\omega_2$-suitable pair $H,i$
   and a poset $P$ of Calibre $\omega_1$ such that
    each of the following hold:
   \begin{enumerate}
   \item Martin's Axiom and $\mathfrak c = \omega_2$,

   \item  the space $X_H$ is   locally compact, $0$-dimensional, first-countable,
          and not compact,
          \item
          the one-point compactification of $X_H$
     has
     a $P$-base
\item the space $X_H$ is initially $\omega_1$-compact and normal,
     \item the one-point compactification of $X_H$ is compact, has
     countable tightness,
    and is not sequential.
     \end{enumerate}
   \end{theorem}

\begin{proof}
We have already established  items (1), (2), and (3).
Item (3) implies that the one-point compactification of $X_H$
has countable tightness. Item (5) is an immediate consequence
of items (1)-(4). So it remains to prove item (4). This will
require a forcing proof over the model $V[G]$.
Before we begin, let us notice that:

\begin{fact}
In $V[G]$, if
$S$
is an unbounded subset of $\omega_2$, then the closure
of $S\times \mathbb{C}$ will contain
 $(\omega_2\setminus\alpha)\times \mathbb{C}$ for some
 $\alpha\in \omega_2$.
 \end{fact}

  This follows from the property in  item (4)
  because of
  the facts that
  $S\times\mathbb{C}$ does not have compact closure
  and that the one-point compactification of $X_H$
  has countable tightness.\bigskip

Recall, from Proposition \ref{mainProp}, that, in $V[G]$,
the closure in $X_H$ of
each countable subset of $X_H$ is either compact
or contains $(\omega_2\setminus\alpha)\times\mathbb{C}$ for some
$\alpha\in \omega_2$. We will prove that this statement remains
true in $V[G][G_2]$. Before doing so we note that item (4)
is a consequence of this claim. It is immediate from (4) that
 $X_H$ is countably compact. The fact that then $X_H$ is initially
 $\omega_1$-compact follows from the fact that a compact
  $P$-space has no converging $\omega_1$-sequences. The fact
  that $X_H$ is normal is noted in \cite{JKS2009}*{\S8} and is similar
  to the proof that an Ostaszewski space is normal.  Indeed,
  it follows from item (4) that for any two   disjoint closed
  subsets of $X_H$ at least one of them is compact.

  Let $\dot A$ be a $Q_{\omega_2}$-name of a countable
  subset of $X_H$.  Assume there is a $q\in G_2$
  such that $q$ forces that the closure of $\dot A$ is not
  compact. Note that  $q$ forces that for all finite $F\subset\omega_2$,
     $\dot A\setminus U[F]$ is not empty. Also, that the closure
     of $\dot A$ is forced to miss $\{\alpha\}\times\mathbb{C}$ if
     and only if $\dot A$ is forced to miss
      $U(\alpha,0)\setminus U[F]$ for some finite $F\subset\alpha$.

   By Proposition \ref{MA},
  there is a
  $\xi<\omega_2$ and a $Q_\xi$-name $\dot B$
  satisfying that
  $\val_{G_2\cap Q_\xi}(\dot B)$ is equal
  to $\val_{G_2}(\dot A)$.
  By possibly choosing a larger value
  of $\xi$, we may assume that $q\in Q_\xi$.
  We first note that it suffices to work with $\dot B$ and
  the poset $Q_\xi$.

  \begin{fact} For each $\lambda\in\omega_2$, $k\in\mathbb {N}$,
  and $t:\{1,\ldots,k\}\mapsto 2$,
  and finite $F\subset\lambda$,
  the following are equivalent:
  \begin{enumerate}
   \item $\val_{G_2}(\dot A)$ misses
   $\bigcap \{U(\lambda,[n,t  (n)]) : 1\leq n\leq k\}
   \setminus U[F]$,
   \item $\val_{G_2\cap Q_\xi}(\dot B)$ misses
   $\bigcap \{U(\lambda,[n,t  (n)]) : 1\leq n\leq k\}
   \setminus U[F]$.
    \end{enumerate}
     \end{fact}

  We must prove that the closure of $\val_{G_2\cap Q_\xi}(\dot B)$
  contains $\{\lambda\}\times \mathbb {C}$ for a co-initial
   set of $\lambda\in
  \omega_2$. This means that
  we are interested in the set of $\lambda\in\omega_2$ such that
 $\{\lambda\}\times\mathbb{C}$ is not contained in the closure
 of $\dot B$. For any such $\lambda$, there must be a
  $q\geq q_\lambda\in Q_\xi$, an
 integer $k_\lambda$ and a function
  $t_\lambda:\{1,\ldots, k_\lambda\}\mapsto 2$, and a finite
   $F_\lambda\subset \lambda$ such that
   $q_\xi$ forces that  $\dot B $ is disjoint
   from $\bigcap \{U(\lambda,[n,t_\lambda (n)]) : 1\leq n\leq k_\lambda\}
   \setminus U[F_\lambda]$.
   Let $S$ denote the set of all $\lambda$ such that
   such a sequence $\langle q_\lambda,k_\lambda,t_\lambda, F_\lambda
   \rangle$ exists.

   If $\alpha\notin S$, then $q$ forces that $\{\alpha\}\times\mathbb{C}$
   is contained in the closure of $\dot B$.  Of course
   this also means that $q$ forces that the closure of $\dot B$
   contains the closure of $(\omega_2\setminus S)\times \mathbb{C}$.
In this case, Fact 1 implies that there is an $\alpha\in\omega_2$
such that the
 closure of $(\omega_2\setminus S)\times\mathbb{C}$,
  and therefore of $\val_{G_2\cap Q_\xi}(\dot B)$~,
will contain $(\omega_2\setminus\alpha)\times \mathbb{C}$ as required.
 Therefore we conclude that
 if $S$ is unbounded, then $\omega_2\setminus S$ is bounded.

 We assume that $S$ is unbounded and obtain a contradiction.
Since $Q_\xi$ has cardinality $\aleph_1$,
 it follows from the pressing down lemma
 that  there is a stationary
 subset $S_1$ of $S$, consisting of limits with cofinality $\omega_1$,
 and a tuple
  $\langle \bar q ,k ,t , F
   \rangle$ such that, for all $\lambda\in S_1$
   \begin{enumerate}
   \item $\bar q = q_\lambda$, $k = k_\lambda$, $t = t_\lambda$, and
   \item $F= F_\lambda$.
   \end{enumerate}
Let $W$ be the union of the family
   $\{ \bigcap \{ U(\lambda,[n,t(n)]):1\leq n\leq k\}\setminus U[F] :\lambda\in
   S_1\}$. Since $W$ is open and
   the property of item (4) holds in $V[G]$, it follows
   that $X_H\setminus W$ is compact. Choose any finite $F_1\subset
   \omega_2$ so that $X_H\subset W\cup U[F_1]$. It now follows
   that $\bar q$ forces that $\dot B$ is contained in $U[F\cup F_1]$,
   which is a contradiction.
   \end{proof}

\textbf{Acknowledgement.} The authors would like to express their gratitude to the referee for all
his/her valuable comments and suggestions which lead to the improvements of the paper.





\end{document}